\documentclass[11pt]{amsart}
\usepackage{amssymb,amscd,amsmath,enumerate}
\usepackage{tikz}
\usetikzlibrary{positioning}
\usepackage[all]{xy}
\newtheorem{thm}{Theorem}[section]

\newtheorem{prop}[thm]{Proposition}

\newtheorem{lem}[thm]{Lemma}

\theoremstyle{definition}

\numberwithin{equation}{section}

\newcommand{\supp}{\mathsf{supp}}

\newcommand{\R}{{\mathbb{R}}}

\newcommand{\N}{{\mathbb{N}}}
\newcommand{\Q}{{\mathbb{Q}}}
\newcommand{\Z}{{\mathbb{Z}}}

\newcommand{\SL}{\mathrm{SL}}

\newcommand{\GeneratingFunctionx}{X}

\newcommand\Tilt{\mathsf{Tilt}}

\def\asup#1#2{\Vert #2\Vert_{#1, \infty}}
\def\aone#1#2{\Vert #2\Vert_{#1, 1}}

\begin{document}
\title[Bounds for $\SL_2$-indecomposables in tensor powers]
{Bounds for $\SL_2$-indecomposables in tensor powers of the natural representation in characteristic $2$}

\author{Michael Larsen}
\email{mjlarsen@indiana.edu}
\address{Department of Mathematics\\
    Indiana University \\
    Bloomington, IN 47405\\
    U.S.A.}

\thanks{The author was partially supported by NSF grant DMS-2001349.}

\begin{abstract}
Let $K$ be an algebraically closed field of characteristic $2$, $G$ be the algebraic group $\SL_2$ over $K$, and  $V$ be the natural representation of $G$.
Let $b_k^{G,V}$ denote the number of $G$-indecomposable factors of $V^{\otimes k}$, counted with multiplicity,
and let $\delta = \frac 32 - \frac{\log 3}{2\log 2}$.  Then there exists a smooth multiplicatively periodic function $\omega(x)$ such that 
$b_{2k}^{G,V} = b_{2k+1}^{G,V}$ is asymptotic to $\omega(k) k^{-\delta}4^k$.  We also prove a lower bound of the form $c_W k^{-\delta}(\dim W)^k$
for $b_k^{G,W} $ for any tilting representation $W$ of $G$.
\end{abstract}

\maketitle

\section{Introduction}

Let $G=\SL_2$ over an algebraically closed field $K$ of characteristic $2$, and let $V$ be the $2$-dimensional natural representation of $G$.
Let $b_k^{G,V}$ denote the number of $G$-indecomposable factors of $V^{\otimes k}$, counted with multiplicity.
Coulembier, Ostrik, and Tubbenhauer ask \cite[Question 6.1]{COT} if there exist $c_1,c_2,\delta > 0$ such that 
\begin{equation}
\label{COT}
c_1 k^{-\delta}2^k\le b_k^{G,V} \le c_2 k^{-\delta}2^k
\end{equation}
They give a heuristic argument, due to Etingof, predicting that $\delta = \frac 32 - \frac{\log 3}{2\log 2}$.
In this paper, we prove that their prediction is right.  Indeed, something stronger is true.
\begin{thm}
\label{Main}
Defining $\delta$ as above,
there exists a smooth function $\omega\colon \R\to (0,\infty)$ such that $\omega(4^x)$ is periodic (mod $1$) and 
$$b_{2k}^{G,V} = b_{2k+1}^{G,V}\sim \omega(k) k^{-\delta}4^k.$$
\end{thm}

We construct the function $\psi(x) = \omega(x)x^{-\delta}$ explicitly 
as an infinite convolution of distributions of the form $\delta_0 + \varphi_s$, where $\delta_0$ is the delta function concentrated at $0$,
and the $\varphi_s$ are rescalings of the theta series associated to
the (mod $4$) primitive Dirichlet character $\chi$.
This theta series is positive and has rapid decay at both $0$ and $\infty$.
These facts follow almost immediately from \eqref{FE}, which is essentially the functional equation of the Dirichlet L-function $L(s,\chi)$.

Our proof interprets $b_{2k}^{G,V}$ as the number of paths of length $k$, counted with multiplicity, of the \emph{fusion graph} of $V^{\otimes 2}$, the directed graph 
giving the decomposition 
into indecomposable factors of the tensor product of a given indecomposable tilting module with $V^{\otimes 2}$.  In the characteristic zero case, 
the vertex set of the graph would be $\N$, and the arrows would connect pairs of consecutive non-negative integers, so we would 
end up counting left factors of Dyck paths of length $k$.  In our characteristic $2$ case, the graph reflects the dyadic nature of the representation category of $G$.
See Figure 1 below (which shows only the component of the graph corresponding to even highest weights.)

Classifying paths of length $k$ according to their final endpoint $n$, we observe two striking departure from the familiar characteristic zero behavior.
Firstly, the number of paths with given $k$ is roughly inversely proportional to $2$ to the power of the number of $1$'s in the binary expansion of $n$.
Second, for fixed $k$, the number of paths terminating in $n$ falls off sharply when $n$ is significantly smaller than $\sqrt k$ as well when $n$ is significantly
larger than $\sqrt k$ (this latter case being in line with characteristic zero behavior).  
This can be regarded as the discrete analogue of the rapid decay of our theta function at $\infty$.
In proving these claims, we are helped greatly by the fact that the generating function $\GeneratingFunctionx_{2^s}(t)$ of paths of length $k$ terminating in $2^s$ satisfies the recursive formula 
\begin{equation}
\label{recurrence}
\GeneratingFunctionx_{2^{s+1}}(t) = \frac{\GeneratingFunctionx_{2^s}(t)^2}{1-2\GeneratingFunctionx_{2^s}(t)^2}.
\end{equation}

I would like to acknowledge a very helpful correspondence with the authors of \cite{CEOT}.  Shortly after answering \cite[Question 6.1]{COT},  I learned that they had independently done so and had, indeed, extended their result to all positive characteristics. 
Our methods were different enough that we agreed it would make sense to write separate papers rather than combine forces.  

An early version of \cite{CEOT} asked whether there exists a continuous function $\omega$ as in Theorem~\ref{Main}.  I realized that my approach 
would give a direct construction of $\omega$.  Their most recent draft gives a non-constructive answer to the same question.

Their paper also asks for asymptotic formulas for tensor powers of tilting modules other than $V$.  In the case $p=2$, this paper gives a lower bound of the form $c_W k^{-\delta}(\dim W)^k$,
but we do not yet have an upper bound of the same form, let alone an asymptotic formula as in Theorem~\ref{Main}.

\section{Tilting representations and the fusion graph}

For every non-negative integer $n$, let $T(n)$ denote the (unique) indecomposable tilting module of $G$ with highest weight $n$.
Thus $V$ is isomorphic to $T(1)$.
Every tensor product of tilting modules is again a tilting module \cite{Donkin} and is therefore determined by its formal character,
which we express as an element of $\Z[t,t^{-1}]$.  

The formal characters of the $T(n)$ are well known.
Following  \cite[Proposition 2.6]{TW}, if
$$n+1 = 2^j + a_{j-1}2^{j-1}+\cdots + a_0,$$
with $a_i\in \{0,1\}$,
we define the \emph{support} of $n$ to consist of all integers $m$ in the set
$$\supp(n) = \{2^j \pm a_{j-1}2^{j-1}\pm\cdots \pm a_0\}.$$
From \cite[Proposition 5.4]{TW}, the formal character of $T(n)$ is
\begin{equation}
\label{FC sum}
\sum_{m\in \supp(n)} \frac{t^m - t^{-m}}{t-t^{-1}},
\end{equation}
where $m$ ranges over the support of $n$.  This can be expressed as
\begin{equation}
\label{FC}
\chi_n(t) = \frac{(t^{2^j}-t^{-2^j})\prod_{\{i\mid a_i=1\}} (t^{2^i} + t^{-2^i})}{t-t^{-1}} = \frac{t^{2^j}-t^{-2^j}}{t-t^{-1}} \prod_{\{i\mid a_i=1\}}(t^{2^i} + t^{-2^i}).
\end{equation}
From this, it follows immediately that if $n+1$ is even, $\chi_1(t)\chi_n(t) = \chi_{n+1}(t)$, so
$$V\otimes T(n)\cong T(n+1).$$
By induction on $r\ge 1$,
$$(t+t^{-1}) \prod_{i=0}^{r-1} (t^{2^i}+t^{-2^i}) = (t^{2^r}+t^{-2^r}) + 2\sum_{i=1}^{r-1}\prod_{j=1}^{i-1}(t^{2^j}+t^{-2^j})$$
From this it follows that if $2^r$ is the highest power of $2$ dividing $n+1$,
$$V^{\otimes 2}\otimes T(2n) \cong T(2n+2) \oplus \bigoplus_{i=1}^{r+1} T(2n+2 - 2^i)^{\oplus 2},$$
unless $n=2^r-1$, in which case we omit the $T(0)$ terms from the above sum:
$$V^{\otimes 2}\otimes T(2n) \cong V\otimes T(2n+1)\cong T(2n+2) \oplus \bigoplus_{i=1}^r T(2n+2 - 2^i)^{\oplus 2}.$$

Consider the labelled directed graph 
on non-negative integers $n$, where there is an arrow from $n$ to $n+1$ labelled $1$
and arrows labelled $2$ from $n$ to $n+1-2^i$ for $0\le i\le r$, where $r$ is the number of factors of $2$ in $n+1$
except that we omit all arrows leading to $0$.  
The multiplicity $x_{n,k}$ of $T(2n)$ as an indecomposable factor in $V^{\otimes 2k}$ is therefore the sum over all directed paths of length $k$ from $0$ to $n$ of
the product of labels.

Let $X_n = X_n(t) = \sum_k x_{n,k}t^k$ denote the generating function of this sum, so
$b_k^{G,V}$ is the sum over $n$ of the $t^k$ coefficient of $X_n(t)$.  As there are no edges from any vertex to $0$,  we have $X_0(t) = 1$.
For $m\ge 1$,
$$X_n = t\Bigl(X_{n-1} + 2\sum_{i=0}^r X_{2^i+n-1}\Bigr),$$
where $2^r$ is now the highest power of $2$ dividing $n$.  We rewrite this equation
\begin{equation}
\label{X-equations}
X_n = t L_n(X_0,X_1,X_2,\ldots),
\end{equation}
where we define
$$L_n(y_0,y_1,y_2,\ldots) = y_{n-1} + 2\sum_{i=0}^r y_{2^i+n-1}.$$

For $n< 2^s$, $L_n$ is a linear combination of $y_0,y_1,\ldots,y_{2^s-1}$, so the system 
of equations
\begin{equation}
\label{system}
y_i = t L_i(y_0,\ldots,y_{2^s-1}),\ 1\le i< 2^s
\end{equation}
consists of $2^s-1$ equations in $2^s$ variables.  The matrix of this system of homogeneous linear equations has rank $2^s-1$ over $\Q((t))$ because its (mod $t$)
reduction has rank $2^s-1$.  Therefore, the solution set of \eqref{system} over $\Q((t))$ is $1$-dimensional.

For $0<i<2^s$,
$$L_{2^s+i}(y_0,y_1,\ldots) = L_i(y_{2^s},y_{2^s+1},\ldots)$$
so $(X_{2^s},X_{2^s+1},\ldots,X^{2^{s+1}-1})$ is the scalar multiple of
$(X_0,X_1,\ldots,X_{2^s-1})$ by an element of $\Q((t))$ (which is, in fact, $X_{2^s}$, since $X_0=1$.)  Therefore,
\begin{equation}
\label{binary rule}
X_{2^s+i} = X^{2^s}X_i
\end{equation}
for $0\le i<2^s$, which implies the general formula
\begin{equation}
\label{product}
X_{2^{s_1}+\cdots+2^{s_q}} = \prod_{i=1}^q X_{2^{s_i}}
\end{equation}
if $s_1>s_2>\cdots>s_q$.  The sequence $X_n$ is therefore determined by its subsequence as $n$ ranges over powers of $2$.

\begin{prop}
\label{F}
Let $F(x) = x^2-2$, and let $F^{\circ s}(x)$ denote the $s$-fold iterate of $F$.  Then
$$(F^{\circ s}(1/t-2))X_{2^s}(t) = 1.$$
\end{prop}

We remark that the fact that the sequence $X_{2^s}^{-1}$ is obtained by iterating $F$ is equivalent to
\eqref{recurrence}.

\begin{proof}
From Figure 1, 
\begin{equation}
\label{X_1}
X_1(t) = \sum_{i=1}^\infty 2^{i-1} t^i = \frac t{1-2t},
\end{equation}
so
$$X_1(t)\bigl(\frac 1t -2\bigr) = 1,$$
and the proposition holds for $n=0$.
By \eqref{X-equations} and \eqref{binary rule},
$$X_{2^s} = tX_{2^s-1}+2t\sum_{i=0}^s X_{2^s+2^i-1} 
= tX_{2^s-1}+ 2tX_{2^s}\sum_{i=0}^s X_{2^i-1}.$$
so rearranging terms and dividing by $tX_{2^s}X_{2^s-1}$, we obtain
\begin{equation}
\label{X inverse}
\frac 1{X_{2^s}} = \frac 1{t\prod_{i=0}^{s-1} X_{2^i}}-2\sum_{i=0}^s\frac 1{\prod_{j=i}^{s-1}X_{2^j}}.
\end{equation}
Dividing both sides by $X_{2^s}$ and subtracting $2$,
$$\frac 1{X_{2^s}^2}-2 = \frac 1{t\prod_{i=0}^s X_{2^i}}-2\sum_{i=0}^s\frac 1{\prod_{j=i}^sX_{2^j}} - 2 = \frac 1{X_{2^{s+1}}},$$
where the last equality comes from substituting $s+1$ for $s$ in \eqref{X inverse}.
The proposition follows by induction on $s$.
\end{proof}

The proposition justifies identifying the power series $X_{2^s}(t)$ with the rational function $(F^{\circ s}(1/t-2))^{-1}$.

%
%
\vskip 10pt
\begin{figure}
\label{0-9}
\begin{tikzpicture}[node distance = .6cm,
roundnode/.style={circle, draw=black, fill=white,  thick, minimum size=4mm},
]
\node[roundnode] (0) {0};
\node[roundnode] (1) [right=of 0]{1};
\node[roundnode] (2) [right=of 1]{2};
\node[roundnode] (3) [right=of 2]{3};
\node[roundnode] (4) [right=of 3]{4};
\node[roundnode] (5) [right=of 4]{5};
\node[roundnode] (6) [right=of 5]{6};
\node[roundnode] (7) [right=of 6]{7};
\node[roundnode] (8) [right=of 7]{8};
\node[roundnode] (9) [right=of 8]{9};

\draw[->] (0) -- (1) node[midway,above] {1};
\draw[->] (1) -- (2) node[midway,above] {1};
\draw[->] (2) -- (3) node[midway,above] {1};
\draw[->] (3) -- (4) node[midway,above] {1};
\draw[->] (4) -- (5) node[midway,above] {1};
\draw[->] (5) -- (6) node[midway,above] {1};
\draw[->] (6) -- (7) node[midway,above] {1};
\draw[->] (7) -- (8) node[midway,above] {1};
\draw[->] (8) -- (9) node[midway,above] {1};
\draw[->] (3)        to [bend left=45] node [below] (L32) {2} (2);
\draw[->] (5)        to [bend left=45] node [below] (L54) {2} (4);
\draw[->] (7)        to [bend left=45]  node [below] (L76) {2} (6);
\draw[->] (9)        to [bend left=45]  node [below] (L98) {2} (8);
\draw[->] (7)        to [bend left=70] node [below] (L74) {2}  (4);
\path (1) edge [loop above] node {2} (1);
\path (2) edge [loop above] node {2} (2);
\path (3) edge [loop above] node {2} (3);
\path (4) edge [loop above] node {2} (4);
\path (5) edge [loop above] node {2} (5);
\path (6) edge [loop above] node {2} (6);
\path (7) edge [loop above] node {2} (7);
\path (8) edge [loop above] node {2} (8);
\path (9) edge [loop above] node {2} (9);
\end{tikzpicture}
\caption{The fusion graph of $V^{\otimes 2}$}
\end{figure}
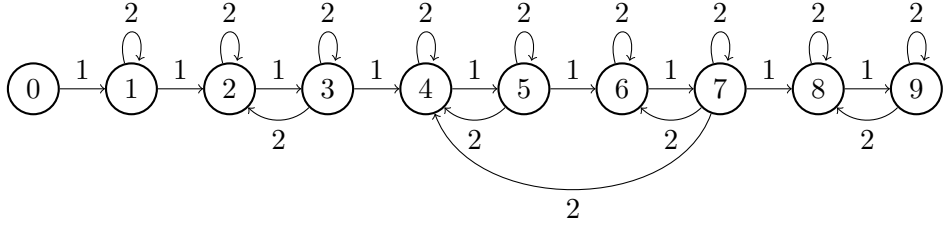

\begin{prop}
\label{x_m formula}
For all $s\ge 0$, we have
\begin{equation}
\label{closed form}
\GeneratingFunctionx_{2^s}(t) = \frac{t^{2^s}}{\prod_{i=1}^{2^s} \bigl(1-\bigl(2+2\cos \frac{(2i-1)\pi}{2^{s+1}}\bigr)t\bigr)}.
\end{equation}
\end{prop}

\begin{proof}
By induction on $s$, $F^{\circ s}(2\cos\theta) = 2 \cos 2^s\theta$, 
so $F^{\circ s}(y-2) = 0$ if $y$ is of the form $2+2\cos\frac{(2i-1)\pi}{2^{s+1}}$, $i=0,1,\ldots,2^{m}-1$.
As $\cos\frac{(2i-1)\pi}{2^{s+1}}$ is strictly decreasing as $i$ ranges from $1$ to $2^s$, this
gives $2^s$ distinct values.
It must include all roots of $F^{\circ s}(y-2) = 0$, since this is a polynomial equation of degree $2^s$.

Thus, 
\begin{align*}
\frac 1{X_{2^s}\bigl(\frac 1t\bigr)} &= \prod_{i=1}^{2^s}  \bigl(\frac 1t-\bigl(2+2\cos \frac{(2i-1)\pi}{2^{s+1}}\bigr)\bigr)\\
&= \frac{\prod_{i=1}^{2^s}  \bigl(1-\bigl(2+2\cos \frac{(2i-1)\pi}{2^{s+1}}\bigr)t\bigr)}{t^{2^s}},
\end{align*}
which implies the proposition.

\end{proof}

\begin{lem}
\label{partial fractions}
If $P(t)$ is a monic polynomial with distinct roots $r_1,\ldots,r_n$, then
$$\frac 1{\prod_{j=1}^n(1-r_jt)} = \sum_{j=1}^n \frac{r_j^{n-1}}{P'(r_j)(1-r_j t)}.$$
\end{lem}

\begin{proof}
Both are rational functions with simple poles at $1/r_1,\ldots,1/r_n$ and no other poles, so it suffices to check that the residues are the same.  Indeed, the residue of the left hand side at the pole $1/r_i$ is
$$\frac {1/r_i}{\prod_{j\neq i} (1-r_j/r_i)} = \frac{r_i^{n-2}}{P'(r_i)},$$
which is the same as that on the right hand side.
\end{proof}

Let 
\begin{equation}
\label{P_s}
P_s(x) = \prod_{i=1}^{2^s} \Bigl(x - \bigl(2+2\cos\frac{(2i-1)\pi}{2^{s+1}}\bigr)\Bigr).
\end{equation}

\begin{lem}
\label{P_m prime}
For any integer $j\in [1,2^s]$,
$$P'_s\Bigl(2+2\cos\frac{(2j-1)\pi}{2^{s+1}}\Bigr) =  \frac{(-1)^{j+1}2^s}{\sin\frac{(2j-1)\pi}{2^{s+1}}}.$$
\end{lem}

\begin{proof}
Let $\zeta = \zeta_{2^{s+2}} = e^{\pi i/2^{s+1}}$.  Then the roots of $P_s(x)$ are $\beta_{2j-1} = 2+\zeta^{2j-1}+\zeta^{1-2j}$
as $j$ ranges from $1$ to $2^s$.  For $j\ge 2$, 
$$\beta_1 - \beta_{2j-1} = \zeta(1-\zeta^{2j-2})(1-\zeta^{-2j}).$$
Since 
$$\prod_{j=1}^{2^{s+1}-1} (1-\zeta^{2j}) = 2^{s+1},$$
we have
$$\prod_{j=2}^{2^s} (\beta_1 - \beta_{2j-1}) = \zeta^{2^s-1}\frac{2^{s+1}}{1-\zeta^{-2}} = \frac{2^{s+1} i}{\zeta-\zeta^{-1}}
= \frac{2^s}{\sin \pi/2^{s+1}},$$
which proves the theorem in the $j=1$ case.  The remaining cases follow by Galois conjugation.
\end{proof}

\begin{prop}
For $k\ge 2^s$, the $t^k$ coefficient of $\GeneratingFunctionx_{2^s}(t)$ is given by
\begin{equation}
\label{GF}
2^{-s}\sum_{j=1}^{2^s} (-1)^{j+1}\sin \frac{(2j-1)\pi}{2^{s+1}}\bigl(2+2\cos \frac{(2j-1)\pi}{2^{s+1}}\bigr)^{k-1}.
\end{equation}
\end{prop}

\begin{proof}
Together Proposition~\ref{x_m formula} and Lemmas \ref{partial fractions} and \ref{P_m prime} imply that
for $k\ge 2^s$, the $t^k$ coefficient of $\GeneratingFunctionx_{2^s}(t)$ is given by
$$\sum_{j=1}^{2^s} a_j\bigl(2+2\cos \frac{(2j-1)\pi}{2^{s+1}}\bigr)^{k-2^s},$$
where
$$a_j = \frac{(-1)^{j+1}\sin \frac{(2j-1)\pi}{2^{s+1}}\bigl(2+2\cos \frac{(2j-1)\pi}{2^{s+1}}\bigr)^{2^s-1}}{2^s}.$$
The proposition follows immediately.
\end{proof}

\section{Discrete convolutions}
The formula \eqref{product} can be understood as expressing the sequence of coefficients $x_{n,k}$ of a general $X_n(t)$ as the convolution of 
the sequences $x_{2^s,k}$ as $s$ ranges over the $s_i$.  As $4^{-k}b_{2k}^{G,V}$ is the sum over all $n$ of $4^{-k}x_{n,k}$, we would like to understand the sum over all finite sets of $s$-values
of the convolutions of the functions $A_s\colon \Z\to \R$ such that $A_s(k)=0$ for $k<0$ and  $A_s(k) = 4^{-k}x_{2^s,k}$ for $k\ge 0$.  

In this section, we analyze such sums more generally.  
We assume each $A_s$ is non-negative, supported on the natural numbers, with sum $1/2$ and with small differences between consecutive terms and finally that
each $A_s$ is concentrated at values of $k$ such that $\log_4 k$ is close to $s$.  It is not difficult to show that our particular functions $A_s$
satisfy these conditions, but we defer this to a later section and work more generally in this section.

If $A\colon \Z\to \R$ is any function and $d$ is a positive integer, we denote by $\asup aA$ (resp. $\aone aA$) the $\ell^\infty$ norm (resp. $\ell^1$ norm) of the restriction of $A$ to $[a,\infty)\cap \Z$.  If $d$ is a positive integer, we denote by $A^d(x)$ the function $A(x+d)-A(x)$.

\begin{lem}
Let $A,B\colon \Z\to [0,\infty)$ be summable and supported on $\N$.  We have
$$\Vert A\ast B\Vert_1 = \Vert A\Vert_1 \Vert B\Vert_1$$
and
$$\Vert A\ast B\Vert_\infty \le \min(\Vert A\Vert_1 \Vert B\Vert_\infty,\Vert A\Vert_\infty \Vert B\Vert_1).$$
\end{lem}

\begin{proof}
For the first claim,
\begin{align*}
\Vert A\ast B\Vert_1 &= \sum_{n=-\infty}^\infty A\ast B(n) = \sum_{n=-\infty}^\infty \sum_{p+q=n} A(p)B(q) \\
&= \sum_{p=-\infty}^\infty\sum_{q=-\infty}^\infty A(p)B(q)
= \Vert A\Vert_1 \Vert B\Vert_1.
\end{align*}
Moreover, for all $n$,
$$A\ast B(n) = \sum_{p=-\infty}^\infty A(p) B(n-p) \le \sum_{p=-\infty}^\infty A(p) \sup_q B(q) = \Vert A\Vert_1 \Vert B\Vert_\infty.$$
By the symmetry of convolution, this implies the second claim.
\end{proof}

\begin{lem}
\label{A split}
Let $A_1,\ldots,A_r$ be functions $\Z\to [0,\infty)$ supported on $\N$, $a_1,\ldots,a_r$ be non-negative integers, and $a$ be an integer greater than or equal to $a_1+\cdots+a_r$.
If $\Vert A_i\Vert_1 = \frac 12$ for $i=1,\ldots,r$, then
$$\asup a{A_1\ast\cdots\ast A_r} \le 2^{1-r}\sum_{i=1}^r \asup {a_i}{A_i}.$$
\end{lem}

\begin{proof}
If $x_0\ge a \ge a_1+\cdots+a_r$, then in any representation of $x_0$ as a sum $x_1+\cdots+x_r$ of integers, the condition $x_i\ge a_i$ must be satisfied for at least one value of $i$.
It suffices to prove that 
$$\sum_{x_i = a_i}^\infty \sum_{x_1+\cdots+x_{i-1}+x_{i+1}+\cdots+x_r = x_0-x_i} A_1(x_1)\cdots A_r(x_r) \le 2^{1-r}\asup {a_i}{A_i}.$$
Since $x_i\ge a_i$, this sum is bounded above by
\begin{equation}
\label{A split part}
\asup {a_i}{A_i}\sum_{x_1,\cdots,x_{i-1},x_{i+1},\cdots,x_r} A_1(x_1)\cdots A_{i-1}(x_{i-1})A_{i+1}(x_{i+1})\cdots A_r(x_r),
\end{equation}
whose $i$th summand is
$$\asup {a_i}{A_i}\Vert A_1\Vert_1\cdots \Vert A_{i-1}\Vert_1\Vert A_{i+1}\Vert_1\cdots \Vert A_r\Vert_1 = 2^{1-r} \asup {a_i}{A_i}.$$
\end{proof}

The following lemma gives explicit form to the principle that the convolution of a rapidly decaying sequence 
with a slowly varying sequence is well approximated by the termwise product of the second sequence with the sum of the first.

\begin{lem}
\label{summing by parts}
If $A,B\colon \Z\to [0,\infty)$ are supported on $\N$ and $a$ and $b$ are non-negative integers, then
%
$$\asup {a+b}{A\ast B-\Vert A\Vert_1 B} \le \Vert A\Vert_1 \sup_{0\le d\le a}   \asup b{B^d} +  2\aone aA \Vert B\Vert_\infty.$$
\end{lem}

\begin{proof}
If $x_0 \ge a+b$, then
\begin{align*}
A\ast B(x_0) &= \sum_{i=0}^a  A(i) (-B(x_0) + B(x_0-i)) + B(x_0)\sum_{i=0}^a A(i) \\
&\qquad\qquad\qquad\qquad+ \sum_{i=a+1}^\infty A(i) B(x_0-i) \\
&= -\sum_{i=0}^a  A(i) B^i(x_0-i) + B(x_0)(\Vert A\Vert_1 - \sum_{i=a+1}^\infty A(i)) \\
&\qquad\qquad\qquad\qquad+ \sum_{i=a+1}^\infty A(i) B(x_0-i).\\
&= -\sum_{i=0}^a  A(i) B^i(x_0-i) + \Vert A\Vert_1B(x_0) \\
&\qquad\qquad\qquad\qquad+  \sum_{i=a+1}^\infty A(i)(B(x_0-i)-B(x_0)).
\end{align*}
We have $\sum_{i=0}^a  A(i) \le \Vert A\Vert_1$, so
the lemma follows.
\end{proof}

Henceforth, we suppose $A_0,A_1,\ldots$ are functions $\Z\to [0,\infty)$ which are supported on $\N$.
For each finite set $S\subset \N$, we denote by $A_S$ the discrete convolution of $A_s$ over all $s\in S$.
For $T$ a subset of $\R$, we define $B_T$ as the sum of $A_S$ over all non-empty finite subsets $S$ of $T\cap \N$.

We make the following assumptions.
\begin{enumerate}
\item[(I)]For all $s$, $\sum_k A_s(k) = 1/2$.
\item[(II)]For all $r\ge 0$ there exists $C_r$ such that $A_s(k) < C_r 4^{-s} (4^s/k)^r$ for all $k>0$.
\item[(III)]There exists $C$ such that for all $s,k_1,k_2$, 
$$|A_s(k_1)-A_s(k_2)| < C(|k_1-k_2| 16^{-s}+8^{-s}).$$
\end{enumerate}

\begin{lem}
\label{estimates}
For $d,m,s\ge 0$ and $S\subset\N$ a finite, non-empty set, the above assumptions imply:
\begin{enumerate}
\item $\Vert A_s\Vert_\infty = O(4^{-s})$.
\item $\asup {2^m}{A_s} = O(2^{6s-4m})$.
\item $\aone{2^m}{A_s} = O(2^{6s-3m})$.
\item $\Vert A_S\Vert_\infty= O(2^{-|S|}2^{-2\max S}).$
\item $\asup {2^m}{A_S} = O(2^{-|S|}2^{6\max S - 4m}).$
\item $\Vert A^d_S\Vert_\infty = O(2^{-|S|}(2^{-3\max S}+2^{-4\max S} d)).$
\end{enumerate}
\end{lem}

\begin{proof}
Parts (1) and (2) follow from assumption (II) in the $r=0$ and $r=4$ cases respectively.

Part (3) follows from assumption (I) if $m\le 2s$.  Otherwise, (2) implies that for $r\ge 0$,
$$\sum_{k=2^{m+r}}^{2^{m+r+1}-1}A_s(k) = O(2^{6s-3m-3r}),$$
and summing over $r$, we get  (3).

For the remaining parts, let $S = \{s_1,\ldots,s_r\}$ with $s_1 > s_2 > \cdots > s_r$.
By (1), $\Vert A_{s_1}\Vert_\infty = O(4^{-s_1})$.  Thus,
$$\Vert A_S\Vert_\infty \le \Vert A_{s_1}\Vert_\infty \Vert A_{s_2}\Vert_1 \cdots \Vert A_{s_r}\Vert_1 = 2^{1-r} \Vert A_{s_1}\Vert_\infty,$$
implying (4).

On the other hand,
$$\sum_{i=1}^r 2^{m-i} < 2^m,$$
so by Lemma~\ref{A split}, (2), and the fact that $s_i \le 1+s_1-i$, we have
$$\asup {2^m}{A_S} \le 2^{1-r} \sum_{i=1}^r \asup {2^{m-i}}{A_{s_i}} = O(2^{-|S|} 2^{6 s_1 -4m}),$$
implying (5).

For (6),
$$A_S^d = A_{s_1}^d\ast A_{s_2} \ast\cdots \ast A_{s_r}.$$
By assumption (III),
$$\Vert A_{s_1}^d\Vert_\infty = O(2^{-3s_1}+ 2^{-4s_1} d).$$
Therefore,
$$\Vert A_S^d\Vert_\infty \le O(2^{1-r}(8^{-s_1}+ 16^{-s_1} d)) = O(2^{-|S|}(2^{-3\max S}+2^{-4\max S} d)).$$

\end{proof}

\begin{lem}
\label{blurring}
If $s\in\N$, $S$ is a finite subset of $(s,\infty)\cap \Z$, $n\ge s$, and $k\in [4^n,4^{n+1})$, then
$$A_{\{s\}\cup S}(k) -\frac 12A_S(k) = O(2^{-|S|}2^{-(\max(\max S+n,2n)+\frac35(n-s))}).$$

\end{lem}

\begin{proof}
We apply Lemma~\ref{summing by parts} with $A= A_s$, $B= A_S$, and $a=2^{n+s}$ and $b=2^{2n-1}$ to obtain
\begin{equation}
\label{L42}
A_{\{s\}\cup S}(k) -\frac 12A_S(k) = O(\sup_{d\le 2^{n+s}} \asup{2^{2n-1}}{A_S^d}+\aone {2^{n+s}}{A_s}\Vert A_S\Vert_\infty).
\end{equation}
Suppose $\max S \ge n$.  Then
$$3\max S \ge \max S + n + \frac 35(n-s)$$
and
$$4\max S - n - s  \ge \max S + n + \frac 35(n-s).$$
Using $\asup{2^{n+s}}{A_S^d}\le \Vert A_S^d\Vert_\infty$
and applying part (6) of Lemma~\ref{estimates},
$$\asup{2^{n+s}}{A_S^d} = O(2^{-|S|}2^{-(\max(\max S+n,2n)+\frac35(n-s))}).$$
As
$$2\max S + 3(n-s)\ge \max S + n + \frac 35(n-s),$$
by parts (3) and (4) of Lemma~\ref{estimates},
$$\aone {2^{n+s}}{A_s}\Vert A_S\Vert_\infty= O(2^{-|S|}2^{-(\max(\max S+n,2n)+\frac35(n-s))}),$$
so the lemma follows from \eqref{L42}.

Likewise, if $n-\frac{n-s}{10}\le \max S \le n$, then we have
$$3\max S,\,4\max S - n - s,\,2\max S + 3(n-s)\ge 2n + \frac 35(n-s),$$
and the lemma follows as before.

If, on the other hand, $\max S\le n-\frac{n-s}{10}$, then since 
$$\asup{2^{2n-1}}{A_S^d}\le 2\asup{2^{2n-1}}{A_S},$$
by part (5) of Lemma~\ref{estimates}, we have
$$\asup{2^{2n-1}}{A_S^d} = O(2^{-|S|}2^{-\frac35(n-s)}2^{-2n}).$$
As $\max S\ge s$,
$$2\max S + 3n-3s \ge 2n+\frac 35(n-s),$$
so by parts (3) and (4) of Lemma~\ref{estimates},
$$\aone {2^{n+s}}{A_s}\Vert A_S\Vert_\infty = O(2^{-|S|}2^{-\frac35(n-s)}2^{-2n}),$$
and again the lemma follows from \eqref{L42}.

%
%
%
%
\end{proof}

\begin{prop}
\label{Limit of B}
For all $\epsilon > 0$, there exists $r>0$ such that if $n\ge r$ is an integer and $k\in [4^n,4^{n+1})$, then
$$\Bigm|\Bigl(\frac 32\bigr)^{-n}B_{\R}(k) - \bigl(\frac 32\bigr)^{-r}B_{[n-r,n+r]}(k) \Bigm|< \epsilon 4^{-n}.$$
\end{prop}

\begin{proof}
By part (4) of Lemma~\ref{estimates}, for $s\ge n$ and $S\subset [0,s)$,
$$A_{\{s\}\cup S}(k)= O(2^{-|S|} 4^{-s}),$$
so
\begin{align*}
\sum_{s=n+r}^\infty \sum_{S\subset [0,s)}A_{\{s\}\cup S}(k) &= O\Bigl(\sum_{s=n+r}^\infty  4^{-s}
\sum_{S\subset [0,s)} 2^{-|S|}\Bigr) \\
&= O\Bigl(\sum_{s=n+r}^\infty  4^{-s}
\bigl(\frac32\bigr)^s\Bigr) 
= O\bigl(\bigl(\frac38\bigr)^{n+r}\bigr).
\end{align*}
%
It follows that
\begin{equation}
\label{too big}
(\frac32\bigr)^{-n}(B_{\R}(k) - B_{[0,n+r]}(k)) = O(4^{-n}(\frac38\bigr)^r\bigr).
\end{equation}
By Lemma~\ref{blurring} and induction on $q$, if $S\subset [n-r,n+r]$ and 
$$n-r>s_1'>\cdots>s'_q,$$
then
$$A_{S\cup \{s'_1,\ldots,s'_q\}}(k) - 2^{-q}A_S(k) =
\begin{cases}
O(2^{-|S|-q}2^{-\frac 35 r}4^{-n})&\text{if $\max S\le n$},\\
O(2^{-|S|-q}2^{-\frac 35 r}2^{\max S - n}4^{-n})&\text{if $\max S>n$}
\end{cases}$$
If $s\in [n-r,n]$ and $S\subset [n-r,s)$, then
$$\sum_{S'\subset [0,n-r)}\bigl(A_{\{s\}\cup S\cup S'}(k)-2^{-|S'|}A_{\{s\}\cup S}(k)\bigr) = O\bigl(2^{-|S|}\bigl(\frac 32\bigr)^{n-r}2^{-\frac35 r}4^{-n}\bigr),$$
so
$$\sum_{S\subset [n-r,s)}\sum_{S'\subset [0,n-r)}\bigl(A_{\{s\}\cup S\cup S'}(k)-2^{-|S'|}A_{\{s\}\cup S}(k)\bigr)
= O\bigl(\bigl(\frac 32\bigr)^s2^{-\frac35 r}4^{-n}\bigr).$$
Therefore,
\begin{equation}
\label{first half}
\begin{split}
\bigl(\frac 32\bigr)^{-n} &\sum_{s=n-r}^n \sum_{S\subset [n-r,s)}\sum_{S'\subset [0,n-r)}\bigl(A_{\{s\}\cup S\cup S'}(k)-2^{-|S'|}A_{\{s\}\cup S}(k)\bigr)\\
&=O(2^{-\frac35 r}4^{-n}).
\end{split}
\end{equation}
If $s\in (n,n+r]$ and $S\subset [n-r,s)$, then
$$\sum_{S'\subset [0,n-r)}\bigl(A_{\{s\}\cup S\cup S'}(k)-2^{-|S'|}A_{\{s\}\cup S}(k)\bigr) = O\bigl(2^{-|S|}\bigl(\frac 32\bigr)^{n-r}2^{-\frac35 r}2^{n-s}4^{-n}\bigr),$$
so
$$\sum_{S\subset [n-r,s)}\sum_{S'\subset [0,n-r)}\bigl(A_{\{s\}\cup S\cup S'}(k)-2^{-|S'|}A_{\{s\}\cup S}(k)\bigr)
= O\bigl(\bigl(\frac 32\bigr)^s2^{-\frac35 r}2^{n-s}4^{-n}\bigr).$$
Therefore,
\begin{equation}
\label{second half}
\begin{split}
\bigl(\frac 32\bigr)^{-n} &\sum_{s=n-r}^n \sum_{S\subset [n-r,s)}\sum_{S'\subset [0,n-r)}\bigl(A_{\{s\}\cup S\cup S'}(k)-2^{-|S'|}A_{\{s\}\cup S}(k)\bigr)\\
&=O(2^{-\frac35 r}4^{-n}).
\end{split}
\end{equation}

By \eqref{too big}, \eqref{first half}, and \eqref{second half}, the sum of all terms in $(3/2)^{-n} B_{\R}(k)$ which do not occur in $(3/2)^{-n} B_{[n-r,n+r]}(k)$
can be taken to be an arbitrarily small multiple of $4^{-n}$ by making $r$ sufficiently large.
\end{proof}

\section{Real convolutions}
Let $\chi$ denote the unique primitive (mod $4$) Dirichlet character, and let
\begin{equation}
\label{Dirichlet}
\varphi(x) = \begin{cases}
\frac\pi8\sum_{n=1}^\infty \chi(n) n e^{-\frac{\pi^2 n^2}{16} x}&\text{ if $x>0$,}\\
0&\text{ if $x\le 0$.}\\
\end{cases}
\end{equation}
As \eqref{GF} suggests, the sequence $x_{2^s,k}$ determines a step function which,
after suitable rescaling, converges as $s\to\infty$ to $\varphi$.  See Proposition~\ref{A vs phi} below for the precise statement and proof.
We prove in this section a continuous analogue of Proposition~\ref{Limit of B}
which enables us to define the function $\psi(x)$ of Theorem~\ref{Main} as a limit of sums of convolutions.
There are significant differences between Proposition~\ref{Limit of B} and Proposition~\ref{Convolve}, however.
For one thing, because we want to prove the limit function is essentially multiplicatively periodic, the index set for
the convolutions must be $\Z$ rather than $\N$.  For another, since we want to prove the limit is smooth, we must bound derivatives of all orders. 
Nevertheless, the proofs are formally very similar.

We begin with some basic facts about convolutions of Schwartz functions over $\R$.  
The convolution of any two such functions $\sigma$ and $\tau$ is again a Schwartz function, and the derivative of $\sigma\ast \tau$
is $\sigma'\ast\tau = \sigma\ast \tau'$ (\cite[V, Proposition 1.11]{SS}).  If $\sigma$ and $\tau$ are non-negative and supported on $[0,\infty)$, the same
will be true of $\sigma\ast \tau$.   We define $\asup af$ (resp. $\aone af$) to be the $L^\infty$  norm (resp. $L^1$ norm) of the restriction of $f$ to $[a,\infty)$.

Exactly as in section 3, we have the following lemma:

\begin{lem}
Let $\sigma,\tau\colon \R\to [0,\infty)$ be Schwartz functions supported on $[0,\infty)$.  Then the Schwartz function $\sigma\ast \tau$ satisfies
$$\Vert\sigma\ast\tau\Vert_1 = \Vert \sigma\Vert_1 \Vert \tau\Vert_1$$
and
$$\Vert \sigma\ast \tau\Vert_\infty \le \min(\Vert \sigma\Vert_1 \Vert \tau\Vert_\infty,\Vert \sigma\Vert_\infty \Vert \tau\Vert_1).$$
\end{lem}

\begin{lem}
\label{Conv a infty}
Let $\sigma_1,\ldots,\sigma_r\colon \R\to [0,\infty)$ be Schwartz functions whose support is contained in $[0,\infty)$,
let $a_1,\ldots,a_r$ be non-negative numbers and $a\ge a_1+\cdots +a_r$.  
If $\Vert \sigma_i\Vert_1 = \frac 12$ for $i=1,2,\ldots,r$, then
$$\asup a{\sigma_1\ast\cdots\ast \sigma_r} \le 2^{1-r}\sum_{i=1}^r \asup {a_i}{\sigma_i}.$$
\end{lem}

\begin{proof}
We proceed by induction on $r$, the base case being $r=2$.  In this case, for $x_0 \ge a$, we have
$$\sigma_1\ast \sigma_2(x_0) = \int_0^{a_1}\sigma_1(x)\sigma_2(x_0-x)dx + \int_{a_1}^{x_0} \sigma_1(x)\sigma_2(x_0-x)dx.$$
Since $\sigma_2(x_0-x) \le \asup {a_2}{\sigma_2}$ in the first integral and $\sigma_1(x) \le \asup{a_1}{\sigma_1}$ in the second integral, we have
\begin{equation}
\label{L51}
\asup a{\sigma_1\ast\sigma_2} \le \Vert \sigma_1\Vert_1 \asup{a_2}{\sigma_2} + \asup{a_1}{\sigma_1}\Vert \sigma_2\Vert_1 = \frac 12 \asup{a_2}{\sigma_2}  + \frac 12 \asup{a_1}{\sigma_1}.
\end{equation}
The general case now follows by induction.
\end{proof}

\begin{lem}
\label{integration by parts} 
If $\sigma$ and $\tau$ are non-negative Schwartz functions supported on $[0,\infty)$ and $a,b\ge 0$, then
$$\asup{a+b}{\sigma\ast\tau-\Vert \sigma\Vert_1\tau}\le a\Vert \sigma\Vert_1 \asup b{\tau'}+2\aone a\sigma\Vert \tau\Vert_\infty.$$
\end{lem}

\begin{proof}
If $x_0 \ge a+b$, then
\begin{align*}
\sigma&\ast\tau(x_0) - \tau(x_0)\int_0^\infty \sigma(x)dx = \int_0^\infty \sigma(x) (-\tau(x_0)+\tau(x_0-x))dx \\
&=\int_0^a \sigma(x) (-\tau(x_0)+\tau(x_0-x))dx + \int_a^\infty \sigma(x) (-\tau(x_0)+\tau(x_0-x))dx\\
&\le\Vert \sigma\Vert_1  a\asup b{\tau'}+ \aone a{\sigma}(2\Vert \tau\Vert_\infty)
\end{align*}
since for $x\le a$, $|\tau(x_0-x)-\tau(x_0)| \le a\asup b{\tau'}$ by the mean value theorem.
\end{proof}

Let $\phi$ be a non-negative Schwartz function supported on $[0,\infty)$ with 
\begin{equation}
\label{phi half}\Vert \phi\Vert_1 = \frac 12.
\end{equation}
For all $s\in\Z$, we define 
$$\phi_s(x) = 4^{-s} \phi(4^{-s}x),$$
so $\Vert \phi_s\Vert_1 = \frac 12$.  If $p$ is a non-negative integer, then
\begin{equation}
\label{pth derivative}
\Vert \phi^{(p)}_s\Vert_\infty = 4^{-(p+1)s}\Vert \phi^{(p)}\Vert_\infty.
\end{equation}
For any finite subset $S\subset \Z$, we define $\phi_S$ to be the convolution of $\phi_s$ over all $s\in S$.
If $S = \{s_1,\ldots,s_r\}$, then
$$\phi_S^{(p)} = \phi^{(p)}_{s_1}\ast \phi_{\{s_2,\ldots,s_r\}}.$$

\begin{lem}
\label{phi estimates}
If $p\ge 0$, $j> 2p+2$, and $S\subset\Z$ is finite and non-empty, the above assumptions imply:
\begin{enumerate}
\item $\Vert \phi^{(p)}_s\Vert_\infty = O(2^{-(2p+2)s})$.
\item $\asup {2^m}{\phi^{(p)}_s} = O(2^{(2j-2p-2)s-jm})$.
\item $\aone{2^m}{\phi^{(p)}_s} = O(2^{(2j-2p-2)s-(j-1)m})$.
\item $\Vert \phi^{(p)}_S\Vert_\infty= O(2^{-|S|}2^{-(2p+2)\max S}).$
\item $\asup {2^m}{\phi^{(p)}_S} = O(2^{-|S|}(2^{(2j-2p-2)\max S - jm}+2^{(4j-4p-4)\max S - (2j-1)m})).$
\end{enumerate}
\end{lem}

\begin{proof}
As $\phi^{(p)}(x)$ is bounded, $|4^{(p+1)s} \phi^{(p)}_s(x)| = |\phi^{(p)}(4^{-s} x)| \le C$ for some constant $C$, which gives $\Vert \phi^{(p)}_s\Vert \le C2^{-(2p+2)s}$, implying part (1).

As $\phi^{(p)}(x) x^j$ is bounded, there exists $C$ such that $\phi^{(p)}(x) \le Cx^{-j}$ for all $x$, so
$$\phi^{(p)}_s(x) = 4^{-(p+1)s} \phi(4^{-s} x)\le C 4^{-(p+1)s}  4^{js}x^{-j} \le C2^{(2j-2p-2)s-jm}$$
if $x\ge 2^m$.  This gives part (2).  Furthermore,
$$\int_{2^m}^\infty \phi^{(p)}_s(x)dx \le C2^{(2j-2p-2)s}\int_{2^m}^\infty x^{-j}dx = O(2^{(2p+6)s-(j-1)m}),$$
giving part (3).

Let $S = \{s_1,\ldots,s_r\}$, $s_1>\cdots >s_r$.  Applying (1) for $s=s_1$, part (4) follows from
$$\Vert \phi^{(p)}_S\Vert_\infty \le \Vert \phi^{(p)}_{s_1}\Vert_\infty \Vert \phi_{\{s_2,\ldots,s_r\}}\Vert_1 = 2^{1-r}\Vert \phi^{(p)}_{s_1}\Vert_\infty .$$

Finally, by \eqref{L51},
$$\asup{2^m}{\phi^{(p)}_S} \le 2^{1-r}\asup{2^{m-1}}{\phi^{(p)}_{s_1}}+\aone{2^{m-1}}{\phi^{(p)}_{s_1}}\asup{2^{m-1}}{\phi_{\{s_2,\ldots,s_r\}}}.$$
By (2), the first summand on the right hand side is $O(2^{-|S|}2^{(2j-2p-2)s_1-jm+j})$.
Applying Lemma~\ref{Conv a infty} with $\sigma_i = \phi_{s_{i+1}}$ and $a_i =2^{m-1-i}$
we get that
$$\asup{2^{m-1}}{\phi_{\{s_2,\ldots,s_r\}}} = O\Bigl(2^{2-r}\sum_{i=1}^r 2^{(2j-2p-2)s_{i+1}-j(m-1-i)}\Bigr),$$
and as $(2j-2p-2)s_{i+1}-j(m-i)$ is strictly decreasing as $i$ increases, this is $O(2^{-r}2^{(2j-2p-2)s_2-jm})$.
By (3),
$$\aone{2^{m-1}}{\phi^{(p)}_{s_1}} = O(2^{(2j-2p-2)s_1-(j-1)m}).$$
Together, these estimates give (5).

\end{proof}

\begin{prop}
\label{phi blurring}
If $p$ is a non-negative integer, $s$ is a negative integer, and $S$ is a finite subset of $[s,\infty)$, then on $[1,4]$,
$$\aone 1{\phi^{(p)}_{\{s\}\cup S}(x) - \frac 12\phi^{(p)}_S(x)} = O(2^{-|S|}2^s).$$
\end{prop}

\begin{proof}
We apply Lemma~\ref{integration by parts} with $\sigma = \phi_s$, $\tau = \phi^{(p)}_S$, $a = 2^s$ and $b=\frac 12$ to obtain
\begin{equation}
\label{L52}
\aone1{\phi_{\{s\}\cup S}(x) - \frac 12\phi_S(x)} = O\bigl(2^s \asup{\frac12}{\phi^{(p+1)}_S} + \aone{2^s}{\phi_s}\Vert \phi_S^{(p)}\Vert_\infty\bigr).
\end{equation}
Applying part (3) of Lemma~\ref{phi estimates} with $j=4p+4$, we obtain
$$\aone{2^s}{\phi_s} = O(2^{(2p+3)s}),$$
so part (4) implies
$$\aone{2^s}{\phi_s}\Vert \phi_S^{(p)}\Vert_\infty = O(2^{(2p+3)s}2^{-|S|}2^{-(2p+2)\max S}).$$
If $\max S\ge 0$, we again use (4) to bound $\asup{\frac 12}{\phi^{(p+1)}_S}\le \Vert \phi^{(p+1)}_S\Vert_\infty$,
and \eqref{L52} gives 
\begin{align*}
\aone1{\phi_{\{s\}\cup S}(x) - \frac 12\phi_S(x)} &= O(2^{-|S|}(2^s 2^{-(2p+4)\max S}+2^{(2p+3)s}2^{-(2p+2)\max S})) \\
&= O(2^{-|S|}2^s 2^{-(2p+2)\max S}),
\end{align*}
implying the proposition.
If $\max S\le 0$, we apply (5) with $j=2p+3$
to bound $\asup{\frac 12}{\phi^{(p+1)}_S}$, so using the fact that $\max S\ge s$, \eqref{L52} implies
$$O(2^{-|S|}2^s2^{5 \max S}+2^{(2p+3)s}2^{-|S|}2^{-(2p+2)\max S}) = O(2^{-|S|}2^s),$$
and the proposition again follows.
\end{proof}

For every subset $T\subset \R$, we define $\psi_T$ to be the sum of $\phi_S$ over all non-empty finite subsets $S\subset T\cap \Z$.
When $T\cap \Z$ is finite, $\psi_T+\delta_0$ is the convolution of the distributions $\psi_s+\delta_0$, $t\in T\cap \Z$, where $\delta_0$ is the delta function
concentrated at $0$.
For $r$ a non-negative integer, we define 
\begin{equation}
\label{psi r}
\psi_r = (3/2)^{-r} \psi_{[-r,r]}.
\end{equation}

\begin{prop}
\label{Convolve}
Let $\phi\colon\R\to \R$ be a non-negative Schwartz function supported on $[0,\infty)$ which satisfies \eqref{phi half}.
There exists a unique smooth function $\psi\colon (0,\infty)\to [0,\infty)$ such that 
the sequence $\psi_1,\psi_2,\ldots$ converges uniformly to $\psi$ on every compact subset of $(0,\infty)$.
Moreover,
$$\psi(4x) =  \frac 38 \psi(x).$$
\end{prop}

\begin{proof}
It suffices to prove the existence of $\psi$.  In fact, we prove slightly more, namely for each fixed $p\in\N$,
$(3/2)^{-r_1}\psi^{(p)}_{[-r_1,r_2]}$ converges uniformly on compact subsets of $(0,\infty)$ 
for any sequence of pairs $(r_1,r_2)$ for which both $r_1$ and $r_2$ 
go to $\infty$.
As 
$$\psi_{[-r_1-1,r_2-1]}(x) = 4\psi_{[-r_1,r_2]}(4x),$$
it suffices to prove convergence on the interval $[1,4]$.  Moreover,
\begin{equation}
\label{periodic}
\begin{split}
\psi(4x) &= \lim_{r_1,r_2\to\infty}\bigl(\frac 32\bigr)^{-r_1} \psi_{[-r_1,r_2]}(4x) 
= \lim_{r_1,r_2\to\infty}\frac{\bigl(\frac 32\bigr)\bigl(\frac 32\bigr)^{-r_1-1} \psi_{[-r_1-1,r_2-1]}(x)}{4}\\
&= \frac 38 \lim_{r_1,r_2\to\infty}\bigl(\frac 32\bigr)^{-r_1-1} \psi_{[-r_1-1,r_2-1]}(x) = \frac 38 \psi(x).
\end{split}
\end{equation}

If $p\in\N$ and $x\in [1,4]$, by part (4) of Lemma~\ref{phi estimates}, we have
\begin{align*}
\bigl(\frac 32\bigr)^{-r_1} \psi^{(p)}_{[-r_1,r_2+1]}(x) &- \bigl(\frac 32\bigr)^{-r_1}  \psi^{(p)}_{[-r_1,r_2]} (x)
= \bigl(\frac 32\bigr)^{-r_1} \sum_{S\subset [-r_1,r_2]}\phi_{\{r_2+1\}\cup S}^{(p)}(x) \\
&=O\bigl(2^{-(2p+2)(r_2+1)}\bigl(\frac 32\bigr)^{-r_1}\sum_{S\subset [-r_1,r_2]}2^{-|S|}\bigr)\\
& = O\bigl(2^{-2(r_2+1)}\bigl(\frac 32\bigr)^{-r_1} \bigl(\frac 32\bigr)^{r_1+r_2+1}\bigr)\\
& = O\bigl(\bigl(\frac 83\bigr)^{-r_2}\bigr).
\end{align*}
By Proposition~\ref{phi blurring},
\begin{align*}
\bigl(\frac 32\bigr)^{-r_1-1} \psi^{(p)}_{[-r_1-1,r_2]}(x) &- \bigl(\frac 32\bigr)^{-r_1}  \psi^{(p)}_{[-r_1,r_2]} (x)\\
&=\bigl(\frac 32\bigr)^{-r_1-1}\sum_{S\subset [-r_1,r_2]} \Bigl(\phi^{(p)}_{\{-r_1-1\}\cup S}(x) - \frac 12 \phi^{(p)}_S(x)\Bigr)\\
&= O\bigl(\bigl(\frac 32\bigr)^{-r_1-1} 2^{-r_1-1}\sum_{S\subset [-r_1,r_2]} 2^{-|S|}  \bigr) \\
&=O\bigl(\bigl(\frac 32\bigr)^{r_2} 2^{-r_1-1}\bigr).
\end{align*}

Applying these together, we conclude that $|\psi^{(p)}_{r+1}(x)-\psi^{(p)}_r(x)|$  is bounded above on $[1,4]$ by
an exponentially decaying function of $r$, and the sequence converges uniformly.
\end{proof}

%

\begin{prop}
The function $\varphi(x)$ in \eqref{Dirichlet} satisfies the hypotheses of Proposition~\ref{Convolve}.
\end{prop}

\begin{proof}
The $k$th derivative of $\varphi(x)$ for $x>0$ is 
$$\frac\pi8\sum_{n=1}^\infty \chi(n)(-\frac{\pi^2n^2}{16})^k n e^{-\frac{\pi^2n^2}{16} x},$$
which is asymptotic to $\frac{(-1)^k\pi^{2k+1} e^{-\frac{\pi^2}{16} x}}{2^{4k+3}}$ at $+\infty$.  Therefore, to prove that $\varphi$ is a Schwartz function, it suffices to show that it has a $k$th derivative at $0$ for all $k$, i.e., that
$$\lim_{x\to 0^+} \sum_{n=1}^\infty \chi(n)n^{2k+1} e^{-\frac{\pi^2n^2}{16} x} = 0.$$

By a theorem of de la Vall\'ee Poussin \cite[Theorem 10.6]{MV}, we have
\begin{equation}
\label{FE}
\varphi(x) = 8(\pi x)^{-3/2}\varphi\Bigl(\frac {16}{\pi^2x}\Bigr)
\end{equation}
for $x>0$.  
Repeatedly differentiating this identity, we can express $\varphi^{(n)}(x)$ as
a linear combination of terms of the form $x^{-j}\varphi^{(k)}(16/\pi^2 x)$, where $j\in\{k+3/2,k+5/2,\ldots,2k+3/2\}$
and $k\in \{0,1,\ldots,n\}$.  Each such term has exponential decay at $x=0$ since $\varphi^{(k)}(y)$
has exponential decay at $y=\infty$.

The integral of $\varphi(x)$ over $\R$ is the limit as $a\to 0^+$ of $\int_a^\infty \varphi(x)dx$,
which can be integrated termwise.
Thus,
$$\int_{-\infty}^\infty \varphi(x) dx 
=\frac\pi 8\sum_{n=1}^\infty \frac{16n\chi(n)}{\pi^2 n^2} = \frac{2 L(1,\chi)}{\pi} = \frac 12.$$

Finally, for the positivity of $\varphi$, it suffices by \eqref{FE} to verify it for $x\ge4/\pi$.  This, in turn, follows from the fact that for $m\ge 0$ and $x\ge 4/\pi$,
$$\frac{(4m+3)e^{-\pi^2(4m+3)^2x}}{(4m+1)e^{-\pi^2(4m+1)^2x}} \le 9 e^{-8\pi^2 x} \le e^{-2 \pi} < 1.$$
\end{proof}

The function $\omega(x)$ in Theorem~\ref{Main} is defined to be $\psi(x) x^\delta$.  By \eqref{periodic}, $\omega(4x) = \omega(x)$.

\section{Convergence to $\psi$}
In this section, we prove the main theorem.  
We follow the notation of \S4.  In particular, $\varphi(x)$ will be defined by \eqref{Dirichlet}, $\psi_r$ will be defined as in \eqref{psi r}
where $\phi$ is taken to be $\varphi$, 
and $\psi(x)$ will be the limit of the $\psi_n$, as in Proposition~\ref{Convolve}.  
The key point in the argument is that $A_s = 4^{-k}x_{2^s,k}$ is well approximated by $\varphi_s$,
and the same thing remains true when we compare convolutions of a bounded number of the sequences $A_s$ and the corresponding 
functions $\phi_S$.

For any function $\sigma\colon \R\to \R$, we write $[\sigma](x)$ for the sequence obtained by restricting $\sigma$ to $\Z$.
We use the same notation $\Vert\;\Vert_1$ and $\Vert\;\Vert_\infty$ for norms on $\R$ and $\Z$; which norm is meant in each case should be clear from context.

\begin{lem}
\label{comparing L1}
For any Schwartz function $\sigma\colon \R\to\R$,
$$\Vert [\sigma]\Vert_1 \le \Vert \sigma\Vert_1 + \Vert \sigma'\Vert_1.$$
\end{lem}

\begin{proof}
By the mean value theorem,
$$| \sigma(x) - \sigma(\lfloor x\rfloor) \le \Vert \sigma'\Vert_\infty.$$
Therefore, for every integer $n$,
$$\int_n^{n+1} |\sigma(x) - \sigma(\lfloor x\rfloor)|dx \le \int_n^{n+1} |\sigma'(x)| dx,$$
so summing over $n$,
\begin{align*}
\Bigm|\int_{-\infty}^\infty \sigma(x)dx - \sum_{-\infty}^\infty \sigma(\lfloor x\rfloor) dx\Bigm|
&\le \sum_{n=-\infty}^\infty \int_n^{n+1} |\sigma(x) - \sigma(\lfloor x\rfloor)|dx \\
&\le \sum_{n=-\infty}^\infty \int_n^{n+1}|\sigma'(x)| dx
=\Vert \sigma'\Vert_1.
\end{align*}
\end{proof}

\begin{lem}
\label{compare two}
For any non-negative Schwartz functions $\sigma$ and $\tau$,
$$\Vert [\sigma\ast\tau] - [\sigma]\ast[\tau]\Vert_\infty \le (\Vert \sigma\Vert_1 +\Vert \sigma'\Vert_1)\Vert \tau'\Vert_\infty + \Vert\sigma'\Vert_\infty (\Vert \tau\Vert_1+\Vert \tau'\Vert_1).$$
\end{lem}

\begin{proof}
If $m$ and $n$ are integers and $x\in [0,1]$, then
\begin{align*}
|\sigma(n)&\tau(m-n) - \sigma(n+x)\tau(m-n-x)| \\
&\le \sigma(n)|\tau(m-n)-\tau(m-n+x)| + \tau(m-n+x)|\sigma(n)-\sigma(n-x)| \\
&\le \sigma(n)\int_{m-n}^{m-n+1}|\tau'(x)|dx + \tau(m-n)\int_{n-1}^n |\sigma'(x)|dx.
\end{align*}
Integrating $x$ over $[0,1]$ and then summing $n$ over $\Z$, we get 
$$\Bigm|\sum_{n=-\infty}^\infty \sigma(n)\tau(m-n) - \sigma\ast\tau(m)\Bigm| \le \sum_{n=-\infty}^\infty \sigma(n)\Vert \tau'\Vert_\infty
+ \sum_{n=-\infty}^\infty \tau(m-n)\Vert \sigma'\Vert_\infty,$$
from which the desired inequality follows by Lemma~\ref{comparing L1}.
\end{proof}

\begin{lem}
\label{compare r}
If $\sigma_1,\ldots,\sigma_r$ are non-negative Schwartz functions with $\Vert \sigma_i\Vert_1 +\Vert \sigma'_i\Vert_1\le \frac34$, and $\Vert \sigma'_i\Vert_\infty\le \epsilon$,
then
$$\Vert [\sigma_1\ast\cdots\ast \sigma_r]-[\sigma_1]\ast\cdots\ast [\sigma_r]\Vert_\infty < 3\epsilon.$$
\end{lem}

\begin{proof}
We prove the upper bound $(2r-2)(3/4)^{r-1}\epsilon$  for $r\ge 2$ by induction on $r$, and that implies the claim.
The case $r=2$ follows from  Lemma~\ref{compare two}.
If $r\ge 3$, then 
\begin{align*}
\Vert [\sigma_1\ast\cdots\ast \sigma_r]&-[\sigma_1]\ast\cdots\ast [\sigma_r]\Vert_\infty\\
&\le \Vert [\sigma_1\ast\cdots\ast \sigma_r]-[\sigma_1]\ast[\sigma_2\ast\cdots\ast \sigma_r]\Vert_\infty\\
&\qquad\qquad+\Vert [\sigma_1]\ast ([\sigma_2\ast\cdots\ast\sigma_r]-[\sigma_2]\ast\cdots\ast[\sigma_r])\Vert_\infty\\
&\le \Vert [\sigma_1\ast\cdots\ast \sigma_r]-[\sigma_1]\ast[\sigma_2\ast\cdots\ast \sigma_r]\Vert_\infty\\
&\qquad\qquad+\Vert [\sigma_1]\Vert_1  \Vert([\sigma_2\ast\cdots\ast\sigma_r]-[\sigma_2]\ast\cdots\ast[\sigma_r])\Vert_\infty.
\end{align*}
We have
$$\Vert \sigma_2\ast\cdots\ast\sigma_r\Vert_1 + \Vert(\sigma_2\ast\cdots\ast\sigma_r)'\Vert_1 
\le \Vert (\sigma_2+\sigma'_2)\Vert_1 \Vert \sigma_3\Vert_1\cdots \Vert \sigma_r\Vert_1 \le \bigl(\frac 34\bigr)^{r-1}$$
and
$$\Vert(\sigma_2\ast\cdots\ast\sigma_r)'\Vert_\infty \le \Vert\sigma_2'\Vert_\infty \Vert \sigma_3\Vert_1\cdots \Vert \sigma_r\Vert_1 \le \bigl(\frac 34\bigr)^{r-2}\epsilon.$$
Applying the $r=2$ case to $\sigma_1$ and $\sigma_2\ast\cdots\ast\sigma_r$
and using the induction hypothesis for $\sigma_2,\ldots,\sigma_r$,
$$\Vert [\sigma_1\ast\cdots\ast \sigma_r]-[\sigma_1]\ast\cdots\ast [\sigma_r]\Vert_\infty \le 2\bigl(\frac 34\bigr)^{r-1}\epsilon + (2r-4)\bigl(\frac 34\bigr)^{r-1}\epsilon.$$
\end{proof}

Given positive  integers $r_1\le r_2$, we set 
$$\sigma_1 = \varphi_{r_1},\sigma_2 = \varphi_{r_1+1},\cdots,\sigma_{1+r_2-r_1}=\varphi_{r_2}.$$
We have, therefore, 
$$\Vert \sigma_i \Vert_1 = \frac 12,\ \Vert \sigma_i\Vert_\infty = O(4^{-r_1-i}),\ \Vert \sigma'_i \Vert_1 = O(4^{-r_1-i}),\ \Vert \sigma'_i \Vert_\infty = O(16^{-r_i-i}),$$
where the implicit constants are absolute.  By Lemma~\ref{compare r}, if $r_1$ is sufficiently large,
\begin{equation}
\label{compare convolutions}
\Vert \sigma_1\ast\cdots\ast\sigma_{1+r_2-r_1} - [\sigma_1]\ast\cdots\ast[\sigma_{1+r_2-r_1}]\Vert_\infty  = O(16^{-r_1}).
\end{equation}

%

Let $\tau_i(x)=0$ for $x\le 0$ and 
$$\tau_i(x) = 
2^{m-i-n}\sum_{j=1}^{2^{n+i-m}} (-1)^{j+1}\sin \frac{(2j-1)\pi}{2^{n+i+1-m}}\Bigl(\frac{1+\cos \frac{(2j-1)\pi}{2^{n+i+1-m}}}{2}\Bigr)^x$$
for $x> 0$.
For $x>0$, we can write
$$\sigma_i(x) = 2^{m-i-n}\sum_{j=1}^\infty (-1)^{j+1} \frac{(2j-1)\pi}{2^{n+i+1-m}}e^{-\frac{(2j-1)^2\pi^2}{4^{n+i-m}}x}.$$

\begin{lem}
\label{A prime}
Let $A_1,\ldots,A_r$ and $A'_1,\ldots,A'_r$ are summable functions $\Z\to\R$, then
$$\Vert A_1\ast\cdots\ast A_r -  A'_1\ast\cdots\ast A'_r\Vert_\infty \le \sum_{i=1}^r \Vert A_i-A'_i\Vert_\infty \prod_{j\neq i} \max (\Vert A_j\Vert_1,\Vert A'_j\Vert_1).$$
\end{lem}

\begin{proof}
By the triangle inequality, 
\begin{align*}\Vert A_1\ast\cdots\ast A_r &-  A'_1\ast\cdots\ast A'_r\Vert_\infty \\
&\le \sum_{i=1}^r  \Vert A_1\ast\cdots\ast A_{i-1}\ast(A_i-A'_i)\ast A'_{i+1}\ast\cdots\ast A'_r\Vert_\infty\\
&\le \sum_{i=1}^r  \Vert A_1\ast\cdots\ast A_{i-1}\ast A'_{i+1}\ast\cdots\ast A'_r\Vert_1\Vert A_i-A'_i\Vert_\infty.
\end{align*}
For each $i$,
\begin{align*}
\Vert A_1\ast\cdots\ast A_{i-1}\ast A'_{i+1}\ast\cdots\ast A'_r\Vert_1 &= \prod_{j=1}^{i-1} \Vert A_j\Vert_1\prod_{j=i+1}^r \Vert A'_j\Vert_1\\
&\le \prod_{j\neq i} \max(\Vert A_j\Vert_1,\Vert A'_j\Vert_1).
\end{align*}
\end{proof}

\begin{prop}
\label{A vs phi}
For all $k\ge 0$, we have
$$4^{-k} x_{2^s,k} = \varphi_s(k-1)+O(8^{-s}).$$
\end{prop}

\begin{proof}

Suppose $k\ge 2^{3s/2}$.
For $a>b>0$, the maximum of $e^{-bx}-e^{-ax}$ on $[0,\infty)$ depends only on $a/b$, so we may consider the case $b=1$.  The maximum value is achieved at $\frac{\log a}{a-1}$,
and by l'H\^opital's rule, this value divided by $a-1$ approaches $1/e$ as $a\to 1$.  As
$$\log\Bigl(\frac{1+\cos x}2\Bigr) = -\frac{x^2}4 + O(x^4),$$
for $j\le 2^{s/4}$,
$$\Bigl(\frac{1+\cos\frac{(2j-1)\pi}{2^{s+1}}}2\Bigr)^{k-1} - \exp\Bigl(-\frac{(2j-1)^2\pi^2}{4\cdot 4^{s+1}}(k-1)\Bigr) = O(8^{-s}).$$
Moreover,
$$2^{-s} \sin \frac{(2j-1)\pi}{2^{s+1}} = (2j-1)2^{-2s-1}\pi + O(j^3 2^{-4s}).$$
Therefore,
\begin{align*}
2^{-s}&\sum_{j\le 2^{m/2}}\Bigm| \sin \frac{(2j-1)\pi}{2^{s+1}}\Bigl(\frac{1+\cos\frac{(2j-1)\pi}{2^{s+1}}}2\Bigr)^{k-1}\\
& \qquad\qquad\qquad\qquad\qquad\qquad- \frac{(2j-1)\pi}{2^{s+1}}\exp\Bigl(-\frac{(2j-1)^2\pi^2}{4\cdot 4^{s+1}}(k-1)\Bigr)\Bigm|\\
&=O(8^{-s}).
\end{align*}
Moreover, for $m$ sufficiently large, the sums
$$2^{-s}\sum_{2^{m/2}<j\le 2^s} \sin \frac{(2j-1)\pi}{2^{s+1}}\Bigl(\frac{1+\cos\frac{(2j-1)\pi}{2^{s+1}}}2\Bigr)^{k-1}$$
and
$$\sum_{j> 2^{m/2}}\frac{(2j-1)\pi}{2^{s+1}}\exp\Bigl(-\frac{(2j-1)^2\pi^2}{4\cdot 4^{s+1}}(k-1)\Bigr)$$
both have the properties that each term is less than half of the previous term,
and the initial term is $o(8^{-s})$.  The proposition follows for $k\ge 2^{3s/2}$.

We may therefore assume $3\le 2^s\le k < 2^{3s/2}$.
As $\varphi(x)$ is Schwartz and identically $0$ for $x<0$, it follows that $\varphi(x) = o(x^6)$ as $x\to 0$, so $\varphi((k-1) 4^{-s-2})$ is $o(8^{-s})$.

To prove that $4^{-k}x_{2^s,k}$ is also $o(8^{-s})$, it suffices to show that $X_s(1/4-1/4k)$ is $o(8^{-s})$.
For this, we observe
$$X_1\bigl(\frac 14-\frac 1{4k}\bigr) = \frac{\frac 14-\frac 1{4k}}{\frac 12 + \frac 1{2k}} < \frac 12 - \frac 1k.$$
For all $z\ge 0$,
$$\frac{(\frac 12-\frac 1{4z+2})^2}{1-2 (\frac 12-\frac 1{4z+2})^2} < \frac 12 - \frac 1{z+2},$$
so by induction, 
$$X_{2^i}\bigl(\frac 14-\frac 1{4k}\bigr) < \frac 12 - \frac 1{4^{-i}(k-2)+2}$$
for $0\le i\le s-1$.  If $r$ is the smallest integer such that $4^{-r}(k-2)<1$, then $r = m/4+O(1)$, and
$$X_{2^r}\bigl(\frac 14-\frac 1{4k}\bigr) < \frac 16.$$

This implies 
$$1-2X_{2^i}\bigl(\frac 14-\frac 1{4k}\bigr)^2 \ge \frac{17}{18}$$
for all $i\ge r$, so by induction,
$$X_{2^i}\bigl(\frac14-\frac 1{4k}\bigr) \le (\frac 16)^{2^{i-r}}(\frac{18}{17})^{2^{i-r}-1}  \le (\frac 3{17})^{2^{i-r}} \le 4^{-2^{i-r}}$$
for $i\ge r$.  Applying this for $i=s$, we get a much stronger upper bound than is needed.

\end{proof}

We now define $A_s(k) = 4^{-k}x_{2^s,k}$ and $A'_s = [\varphi_s]$.  As usual, we define $B_T(k)$ to be the sum of $A_S(k)$ over all finite subsets $S\subset T\cap \N$.
Thus, 
$$B_{\R}(k) = 4^{-k} b_{2k}^{G,V}.$$

\begin{prop}
Defining $A_s(k) = 4^{-k}x_{2^s,k}$, the sequence $A_0,A_1,A_2,\ldots$ satisfies properties (I)--(III) of \S2.
\end{prop}

\begin{proof}
We have
$$\sum_{k=0}^\infty A_s(k)t^k = X_{2^s}(t/4).$$
By \eqref{X_1}, $X_1(1/4) = 1/2$.
Using \eqref{recurrence} and induction on $n$, we deduce that  $X_{2^s}(1/4) = 1/2$ for all $s\ge 0$.  This implies (I).

By Proposition~\ref{A vs phi}, (II) follows from the fact that $\varphi(x) = O(x^{-r})$ as $x\to \infty$, and (III) follows from the fact that $|\varphi'(x)|$ is bounded on $\R$.
\end{proof}

We can now prove Theorem~\ref{Main}.

\begin{proof}
As $\varphi_s(k-1)-\varphi_s(k)=O(16^{-s})$.  By Proposition~\ref{A vs phi}, $|A_s(k)-A'_s(k)| = O(8^{-s})$, so Lemma~\ref{A prime} implies 
that for fixed $r$, for $s\ge 4r$, and for $S=\{s_1,s_2,\ldots\}\subset [s-r,r+s]$,
$$\Vert A_S - [\varphi_{s_1}]\ast[\varphi_2]\ast\cdots\Vert_\infty = O(8^{-s}).$$
By \eqref{compare convolutions},
$$\Vert A_S - [\varphi_S]\Vert_\infty = O(8^{-s}),$$
and summing over $S\subset [s-r,s+r]$,
$$\Vert B_{[s-r,s+r]} - [\psi_{[s-r,s+r]}]\Vert\infty = O(2^{-5s/2}).$$
By Proposition~\ref{Limit of B}, for all $\epsilon > 0$, if $r$ is large enough in terms of $\epsilon$,
then for all $k\in [4^s,4^{s+1})$,
$$\bigm|\bigl(\frac 32\bigr)^{-s}B_{\R}(k) - \bigl(\frac 32\bigr)^{-r}\psi_{[s-r,s+r]}(k)\bigm| 
\le \frac\epsilon2 4^{-s}.$$
As $\psi_r$ converges to $\psi$ on $[1,4]$, for $r$ sufficiently large
$$|\psi_r(4^{-s}k) - \psi(4^{-s}k)| \le \frac\epsilon 2,$$
so by \eqref{periodic},
$$\bigm|\bigl(\frac 32\bigr)^{-r}4^s \psi_{[s-r,s+r]}(k) - \bigl(\frac 83\bigr)^s \psi(k)\bigm| \le \frac\epsilon 2.$$
Therefore,
$$\bigm|\bigl(\frac 32\bigr)^{-r} \psi_{[s-r,s+r]}(k) - \bigl(\frac 23\bigr)^s \psi(k)\bigm| \le \frac\epsilon 24^{-s},$$
which implies
$$|B_{\R}(k) -\psi(k)| \le \epsilon \bigl(\frac 83\bigr)^{-s}.$$
As $k<4^{s+1}$, and $\omega(x) = \psi(x) x^\delta = \psi(x) x^{\log_4(8/3)}$ is bounded  away from zero, 
$$\psi(k)\bigl(\frac83\bigr)^{-s} \ge \frac 38\omega(k)$$
implies 
$$\psi(k)^{-1}B_{\R}(k)\to 1$$
as $k\to \infty$.

\end{proof}

\section{A lower bound for powers of any tilting representation}

We conclude by proving that a lower bound of the type predicted in \cite{COT} exists for all tilting representations of $\SL_2$ in characteristic $2$.

\begin{thm}
If $G=\SL_2$ over an algebraically closed field of characteristic $2$ and $W$ is any tilting representation of $G$, then there exists $c_W>0$ such that for all $k\ge 1$,
$$b_k^{G,W} \ge c_W k^{-\delta} (\dim W)^k.$$
\end{thm}

The rest of the section is devoted to the proof of this result.  We begin by observing that the map $Q(x)\mapsto Q(V)$ defines an isomorphism from $\Z[x]$ to $\Tilt(G)$, 
the ring of virtual tilting representations of $G$.  Indeed, the formal character identifies $\Tilt(G)$ with $\Z$-linear combinations of $\chi_n(t)$
as computed in \eqref{FC}.  The $\Z$-linear combinations of the $\chi_n(\Z)$ comprise the ring of $\Z/2\Z$-invariant Laurent polynomials in $t$ with coefficients in $\Z$, where
the non-trivial element of the Weyl group $\Z/2\Z$ of $\SL_2$
maps $t\mapsto t^{-1}$.  It is clear that 
$$Q(x)\mapsto Q(\chi_1(t)) = Q(t+t^{-1})$$
gives an isomorphism $\Z[x]\to \Z[t,t^{-1}]^{\Z/2\Z}$.

\begin{lem}
\label{Q}
If $Q(x)\in \Z[x]$ is such that $W = Q(V)$ is a non-trivial effective representation, then 
\begin{enumerate}
\item $Q(2) = \dim W$, 
\item $Q'(2) > 0$
\item $|Q(x)| < \dim W$ for all $x\in (-2,2)$
\item $|Q(-2)| = \dim W$ if and only if $W$ purely even or purely odd,
i.e., is a direct sum of tilting representations whose highest weights are all even or are all odd.
\end{enumerate}
\end{lem}

\begin{proof}
The dimension of $W$ is obtained by substituting $t=1$ in the formal character $Q(t+t^{-1})$ of $W$, so it is $Q(2)$.

By the chain rule and l'H\^opital's rule,
$$Q'(2) = \lim_{\theta\to 0}\frac{\frac d{d\theta}Q(e^{i\theta}+e^{-i\theta})}{\frac d{d\theta}(e^{i\theta}+e^{-i\theta})} =\frac{\frac{d^2}{d\theta^2}Q(e^{i\theta}+e^{-i\theta})}{-2\cos\theta}\vert_{\theta=0} = \frac{I_2(W)}2,$$   
where $I_2(W)$ denotes Dynkin's representation index, which, for an $\SL_2$ representation with formal character $\sum a_n t^n$ is $\sum_n a_n^2$ \cite[(2.4)]{MP}.  This implies (2).

Since $W$ is non-trivial, $Q(t+t^{-1}) - \chi_n(t)$
has non-negative coefficients for some $n\ge 1$, so by \eqref{FC sum}, either the $1$ and $t^2$ coefficients of $Q(t+t^{-1})$ are both positive, or the
$t$ and $t^{-1}$ coefficients are both positive.  Either way,
$$|Q(e^{i\theta}+e^{-i\theta})| < Q(2)$$
for $0<\theta<\pi$, implying (3).

Finally, when $\theta=\pi$, $|Q(-2)| \le Q(2)$ with equality if and only if all the $m$ for which the $t^m$-coefficient of $Q(t+t^{-1})$ is positive have the same parity.
By \eqref{FC sum}, this occurs if and only if all the highest weights have the same parity.
\end{proof}

For each $k$, the multiplicity of $T(2n)$ as a (virtual) factor of $P(V)$ determines an additive map $\mu_n\colon \Z[x]\to\Z$.
By definition, $\mu_n(x^{2k})=x_{n,k}$ and $\mu_n(x^{2k+1})=0$.

Let $n=2^{s_1}+\cdots+2^{s_r}$ with $s_1>s_2>\cdots>s_r$.  For each integer $s\ge -2$ and integer $j$, let 
$$\beta_{s,j} = \zeta^j_{2^{s+2}}+2+\zeta^{-j}_{2^{s+2}},$$
with the convention that $\beta_{s,j}=4$ for $s<-2$ so that
$$(\beta_{s,j}-2)^2 = \beta_{s-1,j}$$
for all $s,j\in\Z$.
By Proposition \ref{F} and \eqref{closed form},
it follows that
\begin{equation}
\label{Res}
P_{s'}(\beta_{s,j}) = \beta_{s-s',j}-2,
\end{equation}
where $P_{s'}$ is defined  as in \eqref{P_s}.

Let
$$R_s = \{\beta_{s,1},\beta_{s,3},\beta_{s,5},\ldots\},$$
so $R_s$ is the set of roots of $P_s(x)$.

By \eqref{product} and Proposition~\ref{x_m formula}, $x_{n,k}$ is the $t^k$-coefficient of
$$\frac{t^n}{\prod_{i=1}^r \prod_{\beta\in R_{s_i}}(1-\beta t)}.$$
If $P^n$ denotes the product of the $P_{s_i}$, then its roots are contained in the set 
$\{\beta_{s_1,j}\mid j\in\Z\}$.
By Lemma~\ref{partial fractions},
$$X_n(t) = \sum_{i=0}^\infty \sum_{\{\beta\mid P^n(\beta)=0\}}\frac{\beta^{i+n-1}t^{i+n}}{(P^n)'(\beta)}.$$
Therefore, 
$$\mu_n(x^{2k}) = \sum_{\{\beta\mid P^n(\beta)=0\}}\frac{\beta^{k-1}}{(P^n)'(\beta)}.$$
It follows that if $Q(x)$ is a linear combination of even powers of $x$, then
\begin{equation}
\label{Q even}
\mu_n(Q(x)) = \sum_{\{\beta\mid P^n(\beta)=0\}}\frac{Q(\beta^{1/2})}{\beta (P^n)'(\beta)}.
\end{equation}
Since $\mu_n$ vanishes on odd powers of $x$, for general $Q$, we have
\begin{equation}
\label{Q general}
\mu_n(Q(x)) = \frac 12 \sum_{\{\beta\mid P^n(\beta)=0\}}\frac{Q(\beta^{1/2})+Q(-\beta^{1/2})}{\beta (P^n)'(\beta)}.
\end{equation}

\begin{lem}
If $n = 2^{s_1}+\cdots+2^{s_r}$ and $\beta_{s_1,j}$ is a root of $P^n$, then
$$|(P^n)'(\beta_{s_1,j} )|> \frac{2^{-(\log_2 j)^2}}{4j^2}|(P^n)'(\beta_{s_1,1})|.$$
\end{lem}

\begin{proof}
If $j  \le 2^{s-s'}$, then by \eqref{Res},
$$P_{s'}(\beta_{s,j}) = 2\cos \frac{2\pi j}{2^{s+2-s'}} = 2\sin \frac{2\pi(2^{s-s'}-j)}{2^{s+2-s'}} \ge 2(1-j 2^{s'-s})$$
as $\sin x\ge \frac {2x}\pi$ on $[0,\pi/2]$.  For all $s'\neq s$ and $j$ odd, $P_{s'}(\beta_{s,j})$ is not zero and is twice the sine of an integer multiple of $\frac{2\pi}{2^{s+2-s'}}$,
so it is at least $2^{1+s'-s}$ in absolute value.  For a given $j$, there are  $\lfloor \log_2 j\rfloor$ choices of $s'<s$ for which $j\ge 2^{s-s'}$, and 
the lower bounds they give are $2^{-0},2^{-1},\ldots,2^{1-\lfloor\log_2 j\rfloor}$.
From these observations, we see that
$$2^{r-1}\ge \prod_{\{i\mid P_{s_i}(\beta_{s_1,j})\neq 0\}} |P_{s_i}(\beta_{s_1,j})| \ge 2^{r-1} 2^{-(\log_2 j)^2}\prod_{l=1}^\infty (1-2^{-l})\ge 2^{r-3} 2^{-(\log_2 j)^2}.$$
On the other hand, by Lemma~\ref{P_m prime}, if $\beta_{s,j}$ is a root of $P_{s'}(x)$, then
$$|P'_{s'}(\beta_{s,j})| = \frac{2^{s'}}{\sin \frac{j\pi}{2^{s+1}}}.$$
Combining these facts, we see that if $\beta_{s_1,j}$ is a root of $P_{s_i}(x)$, then
$$\frac{2^{s_i}}{\sin \frac{j\pi}{2^{s_1+1}}}2^{r-1}\ge |(P^n)'(\beta_{s_1,j})| \ge \frac{2^{s_i}}{\sin \frac{j\pi}{2^{s_1+1}}}2^{r-3} 2^{-(\log_2 j)^2}.$$
If $\beta_{s_1,j}$ is a root of $P_{s_i}$, then $j$ is divisible by $2^{s_1-s_i}$, so $2^{s_i-s_1}\ge j^{-1}$, and
\begin{equation}
\label{Pnprime}
\frac{|(P^n)'(\beta_{s_1,j})|}{|(P^n)'(\beta_{s_1,1})|} \ge \frac{2^{-(\log_2 j)^2}}{4j^2}.
\end{equation}
\end{proof}

We can now prove the theorem.

\begin{proof}
Let $W = Q(V)$.  We assume first that $W$ is neither even nor odd.  In this case, by Lemma~\ref{Q}, there exist $\epsilon,c_1>0$ such that for all $v\in[4-\epsilon,4]$
and all $u<v$, we have
$$Q(\sqrt v)-Q(\sqrt u) \ge c_1(v-u).$$
We claim that $W$ determines a positive integer $h$ such that for all sufficiently large $s$, all $n\in [2^s,2^{s+1})$, all $k\in [4^{s+h},4^{s+h+1})$, and
all $j\in [1,2^s)$,
\begin{equation}
\label{claim}
\frac{Q(\beta^{1/2}_{s,j})^k}{|\beta_{s,j}(P^n)'(\beta_{s,j})|}\le 4^{1-j}\frac{Q(\beta^{1/2}_{s,1})^k}{|\beta_{s,1}(P^n)'(\beta_{s,1})|}.
\end{equation}
Indeed for $s$ sufficiently large, $\beta_{s,1}>v$, so
\begin{align*}
\frac{Q(\beta^{1/2}_{s,1})^k}{Q(\beta^{1/2}_{s,j})^k} &\ge \bigl(1+\frac{c_1(\beta_{s,1}-\beta_{s,j})}{Q(2)}\bigr)^k  \\
&\ge (1+c_2(\beta_{s,1}-\beta_{s,j}))^{4^{s+h}}\\
&= (1+c_2(4\sin\frac{(j+1)\pi}{2^{s+2}}\sin\frac{(j-1)\pi}{2^{s+2}}))^{4^{s+h}}\\
&\ge (1+c_2\frac{j^2-1}{4^{s+1}})^{4^{s+h}}\ge \exp(c_2(j^2-1)4^{h-1})
\end{align*}
for some $c_2>0$ which does not depend on $j$ or $h$.  For $j\le 2^s$, also $\beta_{s,j}\ge 2$, so 
$$\frac{\beta{s,j}}{\beta_{s,1}} \ge \frac 12.$$
Thus \eqref{claim} follows easily from \eqref{Pnprime} when $h$ is sufficiently large.  By part (3) of Lemma~\ref{Q},
the terms $Q(\beta^{1/2}_{s,j})^k$ for $j\ge 2^s$ and $Q(-\beta^{1/2}_{s,j})^k$ for any $j$
are bounded above by $((1-c_3)\dim W)^k$ for some $c_3 > 0$ which does not depend on $s$, $j$, or $k$.
Therefore,
$$\frac{Q(\pm\beta^{1/2}_{s,j})^k |\beta_{s,1}(P^n)'(\beta_{s,1})|}{Q(\beta^{1/2}_{s,1})^k|\beta_{s,j}(P^n)'(\beta_{s,j})|}$$
is bounded above by a term of the form $(1-c_4)^{4^{s+h}}$ for some $c_4>0$.
We conclude that for $s$ and $h$ sufficiently large, 
\begin{align*}
\mu_n(Q^k) \ge \frac 12 \frac{Q(\beta_{s,1}^{1/2})^k}{\beta_{s,1}(P^n)'(\beta_{s,1})}&\ge  2^{1-r}2^{-s-4}\sin\frac \pi{2^{s+1}}Q(\beta_{s,1}^{1/2})^k\\
&\ge 2^{1-r}2^{-2s-4} Q(\beta_{s,1}^{1/2})^k
\end{align*}
for $k\in [4^{s+h},4^{s+h+1})$.  Summing over all $n\in [2^s,2^{s+1})$, we get
$$b_k^{G,W}\ge \sum_{n=2^s}^{2^{s+1}-1}\mu_n(Q^k)  \ge \bigl(\frac32\bigr)^{s-1}2^{-2s-4}Q(\beta_{s,1}^{1/2})^k = \frac 1{16}\bigl(\frac38\bigr)^sQ(\beta_{s,1}^{1/2})^k.$$
As $Q'(2) > 0$, there exists $c_5>0$ such that $Q(\beta^{1/2}_{s,1}) > \dim W - c_5(4^{-s})$, so for fixed $h$ and $k< 4^{s+h}$, we can bound $4^{-k} Q(\beta_{s,1}^{1/2})^k$ away from $0$.

Finally we consider the cases that $W$ is even or odd.  In the even case (i.e., when $Q(x)$ is an even function) we use \eqref{Q even} instead of \eqref{Q general}, so we sum only over the non-negative square roots of the $\beta_{s,j}$.  Since part (3) of Lemma~\ref{Q} still holds, it remains true that
the terms $Q(\beta^{1/2}_{s,j})^k$ for $j\ge 2^s$ 
are bounded above by $((1-c_3)\dim W)^k$ for some $c_3 > 0$.
For $W$ odd and for even $k$, $Q^k$ is an even function, and we proceed as in the even case.  Finally, for $W$ odd and $k$ odd, we use the fact, true for all groups and all representations, that $b_{k+1}^{G,W} \ge b_k^{G,W}$.

\end{proof}

\end{document}